\documentclass{amsart}%
\usepackage{amsmath}
\usepackage{color}
\usepackage{amssymb}
\usepackage{amsfonts}
\usepackage{graphicx}%
\newtheorem{theorem}{Theorem}[section]
\theoremstyle{plain}

\newtheorem{lemma}[theorem]{Lemma}

\newtheorem{proposition}[theorem]{Proposition}

\numberwithin{equation}{section}

\begin{document}
\title[A supercritical elliptic problem]{Positive solutions to a supercritical elliptic problem\ which concentrate
along a thin spherical hole}
\author{M\'{o}nica Clapp}
\address{Instituto de Matem\'{a}ticas, Universidad Nacional Aut\'{o}noma de M\'{e}xico,
Circuito Exterior, C.U., 04510 M\'{e}xico D.F., Mexico}
\email{monica.clapp@im.unam.mx}
\author{Jorge Faya}
\address{Instituto de Matem\'{a}ticas, Universidad Nacional Aut\'{o}noma de M\'{e}xico,
Circuito Exterior, C.U., 04510 M\'{e}xico D.F., Mexico}
\email{jorgefaya@gmail.com}
\author{Angela Pistoia}
\address{Dipartimento di Metodi e Modelli Matematici, Universit\'{a} di Roma "La
Sapienza", via Antonio Scarpa 16, 00161 Roma, Italy}
\email{pistoia@dmmm.uniroma1.it}
\thanks{M. Clapp and J. Faya are supported by CONACYT grant 129847 and PAPIIT grant
IN106612 (Mexico). A. Pistoia is supported by Universit\`{a} degli Studi di
Roma "La Sapienza" Accordi Bilaterali "Esistenza e propriet\`{a} geometriche
di soluzioni di equazioni ellittiche non lineari" (Italy).}
\date{March 3, 2013}
\maketitle

\begin{abstract}
We consider the supercritical problem%
\[
-\Delta v=\left\vert v\right\vert ^{p-2}v\quad\text{in }\Theta_{\epsilon
},\qquad v=0\quad\text{on }\partial\Theta_{\epsilon},
\]
where $\Theta$ is a bounded smooth domain in $\mathbb{R}^{N},$ $N\geq3,$
$p>2^{\ast}:=\frac{2N}{N-2},$ and $\Theta_{\epsilon}$ is obtained by deleting
the $\epsilon$-neighborhood of some sphere which is embedded in $\Theta$. In
some particular situations we show that, for $\epsilon>0$ small enough, this
problem has a positive solution $v_{\epsilon}$ and that these solutions
concentrate and blow up along the sphere as $\epsilon\rightarrow0$.

Our approach is to reduce this problem to a critical problem of the form%
\[
-\Delta u=Q(x)\left\vert u\right\vert ^{\frac{4}{n-2}}u\quad\text{in }%
\Omega_{\epsilon},\qquad u=0\quad\text{on }\partial\Omega_{\epsilon},
\]
in a punctured domain $\Omega_{\epsilon}:=\{x\in\Omega:\left\vert x-\xi
_{0}\right\vert >\epsilon\}$ of lower dimension, by means of some Hopf map. We
show that, if $\Omega$ is a bounded smooth domain in $\mathbb{R}^{n},$
$n\geq3$, $\xi_{0}\in\Omega,$ $Q\in C^{2}(\overline{\Omega})$ is positive and
$\nabla Q(\xi_{0})\neq0$ then, for $\epsilon>0$ small enough, this problem has
a positive solution $u_{\epsilon},$ and that these solutions concentrate and
blow up at $\xi_{0}$ as $\epsilon\rightarrow0$.\medskip\ 

\medskip


\end{abstract}

\section{Introduction}

We are interested in the supercritical problem%
\begin{equation}
-\Delta v=\left\vert v\right\vert ^{p-2}v\quad\text{in }\mathcal{D},\qquad
v=0\quad\text{on }\partial\mathcal{D}, \label{general}%
\end{equation}
where $\mathcal{D}$ is a bounded smooth domain in $\mathbb{R}^{N},$ $N\geq3,$
and $p>2^{\ast},$ with $2^{\ast}:=\frac{2N}{N-2}$ the critical Sobolev exponent.

Existence of a solution to this problem is a delicate issue. Pohozhaev's
identity \cite{po} implies that (\ref{general}) does not have a nontrivial
solution if $\mathcal{D}$ is strictly starshaped and $p\geq2^{\ast}$. On the
other hand, Kazdan and Warner \cite{kw} showed that infinitely many radial
solutions exist for every $p\in(2,\infty)$ if $\mathcal{D}$ is an annulus. For
$p=2^{\ast}$ Bahri and Coron \cite{bc} established the existence of at least
one positive solution to problem (\ref{general}) in every domain $\mathcal{D}$
having nontrivial reduced homology with $\mathbb{Z}/2$-coefficients. However,
in the supercritical case this is not enough to guarantee existence. In fact,
for each $1\leq k\leq N-3,$ Passaseo \cite{pa1,pa2} exhibited domains having
the homotopy type of a $k$-dimensional sphere in which problem (\ref{general})
does not have a nontrivial solution for $p\geq2_{N,k}^{\ast}:=\frac
{2(N-k)}{N-k-2}.$ Existence may fail even in domains with richer topology, as
shown in \cite{cfp}.

The first nontrivial existence result for $p>2^{\ast}$ was obtained by del
Pino, Felmer and Musso \cite{dfm} in the slightly supercritical case, i.e. for
$p>2^{\ast}$ but close enough to $2^{\ast}.$ For $p$ slightly below
$2_{N,1}^{\ast}$ solutions in certain domains, concentrating at a boundary
geodesic as $p\rightarrow2_{N,1}^{\ast},$ were constructed in \cite{dmp}.

A fruitful approach to produce solutions to the supercritical problem
(\ref{general}) is to reduce it to some critical or subcritical problem in a
domain of lower dimension, either by considering rotational symmetries, or by
means of maps which preserve the laplacian, or by a combination of both. This
approach has been recently taken in \cite{acp, cfp, kp, kp2, pp, wy} to
produce solutions of (\ref{general}) in different types of domains. We shall
also follow this approach to obtain a new type of solutions in domains with
thin spherical perforations.

We start with some notation. Let $O(N)$ be the group of linear isometries of
$\mathbb{R}^{N}$. If $\Gamma$ is a closed subgroup of $O(N)$, we denote by
$\Gamma x:=\{gx:g\in\Gamma\}$ the $\Gamma$-orbit of $x\in\mathbb{R}^{N}$. A
domain $\mathcal{D}$ in $\mathbb{R}^{N}$ is called $\Gamma$-invariant if
$\Gamma x\subset\mathcal{D}$ for all $x\in\mathcal{D},$ and a function
$u:\mathcal{D}\rightarrow\mathbb{R}$ is called $\Gamma$-invariant if $u$ is
constant on every $\Gamma x.$ We denote by
\[
\mathcal{D}^{\Gamma}:=\{x\in\mathcal{D}:gx=x\text{ \ }\forall g\in\Gamma\}
\]
the set of $\Gamma$-fixed points in $\mathcal{D}$.

We consider the problem%
\[
(\wp_{Q,\epsilon}^{\ast})\qquad\left\{
\begin{array}
[c]{ll}%
-\Delta u=Q(x)u^{\frac{n+2}{n-2}} & \text{in }\Omega_{\epsilon},\\
u>0 & \text{in }\Omega_{\epsilon},\\
u=0 & \text{on }\partial\Omega_{\epsilon},
\end{array}
\right.
\]
in
\[
\Omega_{\epsilon}:=\{x\in\Omega:\left\vert x-\xi_{0}\right\vert >\epsilon\},
\]
where $n\geq3,$ $\Omega$ is a bounded smooth domain in $\mathbb{R}^{n}$ which
is invariant under the action of some closed subgroup $\Gamma$ of $O(n),$
$\xi_{0}\in\Omega^{\Gamma},$ and the function $Q\in C^{2}(\overline{\Omega})$
is $\Gamma$-invariant and satisfies $\min_{x\in\overline{\Omega}}Q(x)>0.$ Note
that, since $\xi_{0}\in\Omega^{\Gamma},$ $\Omega_{\epsilon}$ is also $\Gamma$-invariant.

We will prove the following result.

\begin{theorem}
\label{ts}Assume that $\nabla Q(\xi_{0})\neq0.$ Then there exists
$\epsilon_{0}>0$ such that, for each $\epsilon\in(0,\epsilon_{0}),$ problem
$(\wp_{Q,\epsilon}^{\ast})$ has a $\Gamma$-invariant solution $u_{\epsilon}$
which concentrates and blows up at the point $\xi_{0} $ as $\epsilon\to0.$
\end{theorem}

Now we describe two situations where one can apply this result to obtain
solutions of supercritical problems which concentrate and blow up at a sphere.

For $N=2,4,8,16$ we write $\mathbb{R}^{N}\mathbb{=K}\times\mathbb{K}$, where
$\mathbb{K}$ is either the real numbers $\mathbb{R}$, the complex numbers
$\mathbb{C}$, the quaternions $\mathbb{H}$ or the Cayley numbers $\mathbb{O}.$
The set of units $\mathbb{S}_{\mathbb{K}}:=\{\vartheta\in\mathbb{K}:\left\vert
\vartheta\right\vert =1\},$ which is a group if $\mathbb{K=R}$, $\mathbb{C}$
or $\mathbb{H}$ and a quasigroup with unit if $\mathbb{K=O}$, acts on
$\mathbb{R}^{N}$ by multiplication on each coordinate, i.e. $\vartheta
(z_{1},z_{2}):=(\vartheta z_{1},\vartheta z_{2}).$ The orbit space of
$\mathbb{R}^{N}$ with respect to this action turns out to be $\mathbb{R}%
^{\dim\mathbb{K}+1}$ and the projection onto the orbit space is the Hopf map
$\mathfrak{h}_{\mathbb{K}}:\mathbb{R}^{N}=\mathbb{K}\times\mathbb{K}%
\rightarrow\mathbb{R}\times\mathbb{K}=\mathbb{R}^{\dim\mathbb{K}+1}$ given by%
\[
\mathfrak{h}_{\mathbb{K}}(z_{1},z_{2}):=(\left\vert z_{1}\right\vert
^{2}-\left\vert z_{2}\right\vert ^{2},\,2\overline{z_{1}}z_{2})\text{.}%
\]
What makes this map special is that it preserves the laplacian. Maps with this
property are called harmonic morphisms \cite{bw,er,w}. More precisely, the
following statement holds true. It can be derived by straightforward
computation (cf. Proposition \ref{ps+}) or from the general theory of harmonic
morphisms as in \cite{cfp}.

\begin{proposition}
\label{propHopf}Let $N=2,4,8,16$ and let $\mathcal{D}$ be an $\mathbb{S}%
_{\mathbb{K}}$-invariant bounded smooth domain in $\mathbb{R}^{N}%
\mathbb{=K}^{2}$ such that $0\notin\overline{\mathcal{D}}.$ Set $\mathcal{U}%
:=\mathfrak{h}_{\mathbb{K}}(\mathcal{D})$. If $u$ is a solution to problem%
\begin{equation}
\left\{
\begin{array}
[c]{ll}%
-\Delta u=\frac{1}{2\left\vert x\right\vert }\left\vert u\right\vert ^{p-2}u &
\text{\emph{in} }\mathcal{U},\\
u=0 & \text{\emph{on} }\partial\mathcal{U},
\end{array}
\right.  \label{prob2}%
\end{equation}
then $v:=u\circ\mathfrak{h}_{\mathbb{K}}$ is an $\mathbb{S}_{\mathbb{K}}%
$-invariant solution of problem \emph{(\ref{general})}. Conversely, if $v$ is
an $\mathbb{S}_{\mathbb{K}}$-invariant solution of problem
\emph{(\ref{general})} and $v=u\circ\mathfrak{h}_{\mathbb{K}},$ then $u$
solves \emph{(\ref{prob2})}.
\end{proposition}

We apply this result as follows: Let $N=4,8,16$ and let $\Theta$ be an
$\mathbb{S}_{\mathbb{K}}$-invariant bounded smooth domain in $\mathbb{R}%
^{N}\mathbb{=K}^{2}$ such that $0\notin\overline{\Theta}.$ Fix a point
$z_{0}\in\Theta$ and for each $\epsilon>0$ small enough let%
\[
\Theta_{\epsilon}:=\{z\in\Theta:\text{dist}(z,\mathbb{S}_{\mathbb{K}}%
z_{0})>\epsilon\}
\]
where $\mathbb{S}_{\mathbb{K}}z_{0}:=\{\vartheta z:\vartheta\in\mathbb{S}%
_{\mathbb{K}}\}.$ This is again an $\mathbb{S}_{\mathbb{K}}$-invariant bounded
smooth domain in $\mathbb{K}^{2}.$ We consider the supercritical problem%
\[
(\wp_{\epsilon}^{1})\qquad\left\{
\begin{array}
[c]{ll}%
-\Delta v=v^{\frac{\dim\mathbb{K}+3}{\dim\mathbb{K}-1}} & \text{in }%
\Theta_{\epsilon},\\
v>0 & \text{in }\Theta_{\epsilon},\\
u=0 & \text{on }\partial\Theta_{\epsilon}.
\end{array}
\right.
\]
Then, Theorem \ref{ts} with $n:=\dim\mathbb{K}+1,$ $\Gamma=\{1\}$,
$\Omega:=\mathfrak{h}_{\mathbb{K}}(\Theta),$ $\xi_{0}:=\mathfrak{h}%
_{\mathbb{K}}(z_{0})$ and $Q(x):=\frac{1}{2\left\vert x\right\vert },$
together with Proposition \ref{propHopf},\ immediately yields the following result.

\begin{theorem}
\label{thmsup1}There exists $\epsilon_{0}>0$ such that, for each $\epsilon
\in(0,\epsilon_{0}),$ the supercritical problem $(\wp_{\epsilon}^{1})$\ has an
$\mathbb{S}_{\mathbb{K}}$-invariant solution $v_{\epsilon}$ which concentrates
and blows up along the sphere $\mathbb{S}_{\mathbb{K}}z_{0}$ as $\epsilon
\rightarrow0$.
\end{theorem}

Now let $O(m)\times O(m)$ act on $\mathbb{R}^{2m}\equiv\mathbb{R}^{m}%
\times\mathbb{R}^{m}$ in the obvious way and $O(m)$ act on the last $m$
coordinates of $\mathbb{R}^{m+1}\equiv\mathbb{R}\times\mathbb{R}^{m}.$ We
write the elements of $\mathbb{R}^{2m}$ as $(y_{1},y_{2})$ with $y_{i}%
\in\mathbb{R}^{m}\ $and the elements of $\mathbb{R}^{m+1}$ as $x=(t,\zeta)$
with $t\in\mathbb{R},$ $\zeta\in\mathbb{R}^{m}.$ Recently Pacella and Srikanth
showed that the real Hopf map provides a one-to-one correspondence between
$\left[  O(m)\times O(m)\right]  $-invariant solutions of a supercritical
problem in a domain in $\mathbb{R}^{2m}$ and $O(m)$-invariant solutions of a
critical problem in some domain in $\mathbb{R}^{m+1}$. In \cite{ps} they
proved the following result.

\begin{proposition}
\label{ps}Let $N=2m,$ $m\geq2,$ and $\mathcal{D}$ be an $\left[  O(m)\times
O(m)\right]  $-invariant bounded smooth domain in $\mathbb{R}^{2m}$ such that
$0\notin\overline{\mathcal{D}}.$ Set
\[
\mathcal{U}:=\{(t,\zeta)\in\mathbb{R}\times\mathbb{R}^{m}:\mathfrak{h}%
_{\mathbb{R}}(\left\vert y_{1}\right\vert ,\left\vert y_{2}\right\vert
)=(t,\left\vert \zeta\right\vert )\text{ for some }(y_{1},y_{2})\in
\mathcal{D}\}.
\]
If $u(t,\zeta)=\mathfrak{u}(t,\left\vert \zeta\right\vert )$ is an
$O(m)$-invariant solution of problem%
\begin{equation}
\left\{
\begin{array}
[c]{ll}%
-\Delta u=\frac{1}{2\left\vert x\right\vert }\left\vert u\right\vert ^{p-2}u &
\text{\emph{in} }\mathcal{U},\\
u=0 & \text{\emph{on} }\partial\mathcal{U},
\end{array}
\right.  \label{prob3}%
\end{equation}
then $v(y_{1},y_{2}):=\mathfrak{u}(\mathfrak{h}_{\mathbb{R}}(\left\vert
y_{1}\right\vert ,\left\vert y_{2}\right\vert ))$ is an $\left[  O(m)\times
O(m)\right]  $-invariant solution of problem \emph{(\ref{general}).}

Conversely, if $v(y_{1},y_{2})=\mathfrak{v}(\left\vert y_{1}\right\vert
,\left\vert y_{2}\right\vert )$ is an $\left[  O(m)\times O(m)\right]
$-invariant solution of problem \emph{(\ref{general}) }and $\mathfrak{v}%
=\mathfrak{u}\circ\mathfrak{h}_{\mathbb{R}},$ then $u(t,\zeta)=\mathfrak{u}%
(t,\left\vert \zeta\right\vert )$ is an $O(m)$-invariant solution of problem
\emph{(\ref{prob2})}.
\end{proposition}

We apply this result as follows: Let $\Phi$ be an $\left[  O(m)\times
O(m)\right]  $-invariant bounded smooth domain in $\mathbb{R}^{2m}$ such that
$0\notin\overline{\Phi}$ and $(y_{0},0)\in\Phi.$ We write $S_{0}%
^{m-1}:=\{(y,0):\left\vert y\right\vert =\left\vert y_{0}\right\vert \}$ for
the $\left[  O(m)\times O(m)\right]  $-orbit of $(y_{0},0),$ and for each
$\epsilon>0$ small enough we set%
\[
\Phi_{\epsilon}:=\{x\in\Phi:\text{dist}(x,S_{0}^{m-1})>\epsilon\}.
\]
This is again an $\left[  O(m)\times O(m)\right]  $-invariant bounded smooth
domain in $\mathbb{R}^{2m}.$ We consider the supercritical problem%
\[
(\wp_{\epsilon}^{2})\qquad\left\{
\begin{array}
[c]{ll}%
-\Delta v=v^{\frac{m+3}{m-1}} & \text{in }\Phi_{\epsilon},\\
v>0 & \text{in }\Phi_{\epsilon},\\
u=0 & \text{on }\partial\Phi_{\epsilon}.
\end{array}
\right.
\]
Then, Theorem \ref{ts} with $n=m+1,$ $\Gamma=O(m),$
\[
\Omega:=\{(t,\zeta)\in\mathbb{R}\times\mathbb{R}^{m}:\mathfrak{h}_{\mathbb{R}%
}(\left\vert y_{1}\right\vert ,\left\vert y_{2}\right\vert )=(t,\left\vert
\zeta\right\vert )\text{ for some }(y_{1},y_{2})\in\Phi\},
\]
$\xi_{0}:=(\left\vert y_{0}\right\vert ,0,\ldots,0)$ and $Q(x)=\frac
{1}{2\left\vert x\right\vert },$ together with Proposition \ref{ps}%
,\ immediately yields the following result.

\begin{theorem}
There exists $\epsilon_{0}>0$ such that, for each $\epsilon\in(0,\epsilon
_{0}),$ problem $(\wp_{\epsilon})$ has an $\left[  O(m)\times O(m)\right]
$-invariant solution $v_{\epsilon}$ which concentrates and blows up along the
$(m-1)$-dimensional sphere $S_{0}^{m-1} $ as $\epsilon\rightarrow0$.
\end{theorem}

The proof of Theorem \ref{ts} uses the well-known Ljapunov-Schmidt reduction,
adapted to the symmetric case. In the following section we sketch this
reduction, highlighting the places where the symmetries play a role. In
section 3 we give an expansion of the reduced energy functional and use it to
prove Theorem \ref{ts}. We conclude with some remarks concerning Proposition
\ref{ps}.

\section{The finite dimensional reduction}

For every bounded domain $\mathcal{U}$ in $\mathbb{R}^{n}$ we take
\[
(u,v):=\int_{\mathcal{U}}\nabla u\cdot\nabla v,\qquad\left\Vert u\right\Vert
:=\left(  \int_{\mathcal{U}}\left\vert \nabla u\right\vert ^{2}\right)
^{1/2},
\]
as the inner product and its corresponding norm in $H_{0}^{1}(\mathcal{U}).$
If we replace $\mathcal{U}$ by $\mathbb{R}^{n}$ these are the inner product
and the norm in $D^{1,2}(\mathbb{R}^{n}).$ We write
\[
\left\Vert u\right\Vert _{r}:=(\int_{\mathcal{U}}\left\vert u\right\vert
^{r})^{1/r}%
\]
for the norm in $L^{r}(\mathcal{U})$, $r\in\lbrack1,\infty).$

If $\mathcal{U}$\ is $\Gamma$-invariant for some closed subgroup $\Gamma$ of
$O(n)$ we set%
\[
H_{0}^{1}(\mathcal{U})^{\Gamma}:=\{u\in H_{0}^{1}(\mathcal{U}):u\text{ is
}\Gamma\text{-invariant}\}
\]
and, similarly, for $D^{1,2}(\mathbb{R}^{n})^{\Gamma}$ and $L^{r}%
(\mathcal{U})^{\Gamma}.$

It is well known that the standard bubbles
\[
U_{\delta,\xi}(x)=[n(n-2)]^{\frac{n-2}{4}}\frac{\delta^{\frac{n-2}{2}}%
}{(\delta^{2}+|x-\xi|^{2})^{\frac{n-2}{2}}}\qquad\delta\in(0,\infty),\text{
\ }\xi\in\mathbb{R}^{n},
\]
are the only positive solutions of the equation
\[
-\Delta U=U^{p}\quad\text{in \ }\mathbb{R}^{n},
\]
where $p:=\frac{n+2}{n-2}.$ Thus, the function $W_{\delta,\xi}:=\gamma
_{0}U_{\delta,\xi},$ with $\gamma_{0}:=[Q(\xi_{0})]^{\frac{-1}{p-1}}$, solves
the equation
\begin{equation}
-\Delta W=Q(\xi_{0})W^{p}\text{\qquad in \ }\mathbb{R}^{n}.\label{ecgamaU}%
\end{equation}
Let
\begin{align}
\psi_{\delta,\xi}^{0} &  :=\frac{\partial U_{\delta,\xi}}{\partial\delta
}=\alpha_{n}\frac{n-2}{2}\delta^{\frac{n-4}{2}}\frac{|x-\xi|^{2}-\delta^{2}%
}{(\delta^{2}+|x-\xi|^{2})^{n/2}},\label{psi}\\
\psi_{\delta,\xi}^{j} &  :=\frac{\partial U_{\delta,\xi}}{\partial\xi_{j}%
}=\alpha_{n}(n-2)\delta^{\frac{n-2}{2}}\frac{x_{j}-\xi_{j}}{(\delta^{2}%
+|x-\xi|^{2})^{n/2}},\text{\qquad}j=1,\dots,n.\nonumber
\end{align}
The space generated by these $n+1$ functions is the space of solutions to the
problem
\begin{equation}
-\Delta\psi=pU_{\delta,\xi}^{p-1}\psi,\qquad\psi\in D^{1,2}(\mathbb{R}%
^{n}).\label{ecpsi}%
\end{equation}
Note that%
\[
U_{\delta,\xi}\in D^{1,2}(\mathbb{R}^{n})^{\Gamma}\qquad\text{iff}\qquad\xi
\in(\mathbb{R}^{n})^{\Gamma}%
\]
and, similarly, for every $j=0,1,\dots,n,$
\[
\psi_{\delta,\xi}^{j}\in D^{1,2}(\mathbb{R}^{n})^{\Gamma}\qquad\text{iff}%
\qquad\xi\in(\mathbb{R}^{n})^{\Gamma}.
\]

Let $\Omega$ be a $\Gamma$-invariant bounded smooth domain in $\mathbb{R}%
^{n},$ $Q\in C^{2}(\overline{\Omega})$ be positive and $\Gamma$-invariant, and
$\xi_{0}\in\Omega^{\Gamma}$. For $\epsilon>0$ small enough set
\[
\Omega_{\epsilon}:=\{x\in\Omega:\left\vert x-\xi_{0}\right\vert >\epsilon\}.
\]
Consider the orthogonal projection $P_{\epsilon}:D^{1,2}(\mathbb{R}%
^{n})\rightarrow H_{0}^{1}(\Omega_{\epsilon})$, i.e. if $W\in D^{1,2}%
(\mathbb{R}^{n})$ then $P_{\epsilon}W$ is the unique solution to the problem
\begin{equation}
-\Delta\left(  P_{\epsilon}W\right)  =-\Delta W\text{\quad in \ }%
\Omega_{\epsilon},\qquad P_{\epsilon}W=0\text{\quad on \ }\partial
\Omega_{\epsilon}. \label{ecu5}%
\end{equation}
A consequence of the uniqueness is that $P_{\epsilon}W\in H_{0}^{1}%
(\Omega_{\epsilon})^{\Gamma}$ if $W\in D^{1,2}(\mathbb{R}^{n})^{\Gamma}.$

We denote by $G(x,y)$ the Green function of the Laplace operator in $\Omega$
with zero Dirichlet boundary condition and by $H(x,y)$ its regular part, i.e.%
\[
G(x,y)=\beta_{n}\left(  \frac{1}{\left\vert x-y\right\vert ^{n-2}%
}-H(x,y)\right)  ,
\]
where $\beta_{n}$ is a positive constant depending only on $n.$ The following
estimates will play a crucial role in the proof of Theorem \ref{ts}.

\begin{lemma}
\label{lemGeMuPi}Assume that $\delta\rightarrow0$ as $\epsilon\rightarrow0$
and $\epsilon=o(\delta)$ as $\epsilon\rightarrow0$. Fix $\eta\in\mathbb{R}%
^{n},$ set $\xi:=\xi_{0}+\delta\eta,$ and define
\[
R(x):=P_{\epsilon}U_{\delta,\xi}(x)-U_{\delta,\xi}(x)+\alpha_{n}\delta
^{\frac{n-2}{2}}H(x,\xi)+\frac{\alpha_{n}}{\delta^{\frac{n-2}{2}}(1+|\eta
|^{2})^{\frac{n-2}{2}}}\frac{\epsilon^{n-2}}{|x-\xi_{0}|^{n-2}}.
\]
Then there exists a positive constant $c$ such that the following estimates
hold true for every $x\in\Omega\smallsetminus B(\xi_{0},\epsilon)$:
\begin{align*}
\left|  R(x)\right|   &  \leq c\delta^{\frac{n-2}{2}}\left[  \frac
{\epsilon^{n-2}(1+\epsilon\delta^{-n+1})}{|x-\xi_{0}|^{n-2}}+\delta
^{2}+\left(  \frac{\epsilon}{\delta}\right)  ^{n-2}\right]  ,\\
\left|  \partial_{\delta}R(x)\right|   &  \leq c\delta^{\frac{n-4}{2}}\left[
\frac{\epsilon^{n-2}(1+\epsilon\delta^{-n+1})}{|x-\xi_{0}|^{n-2}}+\delta
^{2}+\left(  \frac{\epsilon}{\delta}\right)  ^{n-2}\right]  ,\\
\left|  \partial_{\xi_{i}}R(x)\right|   &  \leq c\delta^{\frac{n}{2}}\left[
\frac{\epsilon^{n-2}(1+\epsilon\delta^{-n})}{|x-\xi_{0}|^{n-2}}+\delta^{2}+
\frac{\epsilon^{n-2}}{\delta^{n-1}} \right]  .
\end{align*}

\end{lemma}

\begin{proof}
See Lemma 3.1 in \cite{Ge-Mu-Pi}.
\end{proof}

For each $\epsilon>0$ and $(d,\eta)\in\Lambda^{\Gamma}:=(0,\infty
)\times(\mathbb{R}^{n})^{\Gamma}$ set (see \eqref{ecgamaU})%
\[
V_{d,\eta}:=P_{\epsilon}W_{\delta,\xi}=\gamma_{0}P_{\epsilon}U_{\delta,\xi
}\text{\qquad with \ }\delta:=d\epsilon^{\frac{n-2}{n-1}},\ \ \xi:=\xi
_{0}+\delta\eta.
\]
The map $(d,\eta)\mapsto V_{d,\eta}$ is a $C^{2}$-embedding of $\Lambda
^{\Gamma}$ as a submanifold of $H_{0}^{1}(\Omega_{\epsilon})^{\Gamma}$, whose
tangent space at $V_{d,\eta}$ is
\[
K_{d,\eta}^{\epsilon}:=\text{span}\{P_{\epsilon}\psi_{\delta,\xi}%
^{j}:j=0,1,\dots,n\}.
\]
Note that, since $\xi_{0},\eta\in(\mathbb{R}^{n})^{\Gamma},$ also $\xi
\in(\mathbb{R}^{n})^{\Gamma}$ and, therefore, $K_{d,\eta}^{\epsilon}\subset
H_{0}^{1}(\Omega_{\epsilon})^{\Gamma}.$ We write%
\[
K_{d,\eta}^{\epsilon,\bot}:=\{\phi\in H_{0}^{1}(\Omega_{\epsilon})^{\Gamma
}:(\phi,P_{\epsilon}\psi_{\delta,\xi}^{j})=0\text{ \ for }j=0,1,\dots,n\}
\]
for the orthogonal complement of $K_{d,\eta}^{\epsilon}$ in $H_{0}^{1}%
(\Omega_{\epsilon})^{\Gamma}$, and $\Pi_{d,\eta}^{\epsilon}:H_{0}^{1}%
(\Omega_{\epsilon})^{\Gamma}\rightarrow K_{d,\eta}^{\epsilon}$ and
$\Pi_{d,\eta}^{\epsilon,\bot}:H_{0}^{1}(\Omega_{\epsilon})^{\Gamma}\rightarrow
K_{d,\eta}^{\epsilon,\bot}$ for the orthogonal projections, i.e.
\[
\Pi_{d,\eta}^{\epsilon}(u):=\sum_{j=0}^{n}(u,P_{\epsilon}\psi_{\delta,\xi}%
^{j})P_{\epsilon}\psi_{\delta,\xi}^{j},\qquad\Pi_{d,\eta}^{\epsilon,\bot
}(u):=u-\Pi_{d,\eta}^{\epsilon}(u).
\]

Let $i_{\epsilon}^{\ast}:L^{\frac{2n}{n+2}}(\Omega_{\epsilon})\rightarrow
H_{0}^{1}(\Omega_{\epsilon})$ be the adjoint operator to the embedding
$i_{\epsilon}:H_{0}^{1}(\Omega_{\epsilon})\hookrightarrow L^{\frac{2n}{n-2}%
}(\Omega_{\epsilon})$, i.e. $v=i_{\epsilon}^{\ast}(u)$ if and only if%
\[
(v,\varphi)=\int_{\Omega_{\epsilon}}u\varphi\qquad\forall\varphi\in
C_{c}^{\infty}(\Omega_{\epsilon})
\]
if and only if
\begin{equation}
-\Delta v=u\text{\quad in }\Omega_{\epsilon},\qquad v=0\text{\quad on
}\partial\Omega_{\epsilon}. \label{ecu2}%
\end{equation}
Sobolev's inequality yields a constant $c>0,$ independent of $\epsilon$, such that%

\begin{equation}
\left\Vert i_{\epsilon}^{\ast}(u)\right\Vert \leq c\left\Vert u\right\Vert
_{\frac{2n}{n+2}}\qquad\forall u\in L^{\frac{2n}{n+2}}(\Omega_{\epsilon
}),\ \ \forall\epsilon>0. \label{ecu1}%
\end{equation}
Note again that
\[
i_{\epsilon}^{\ast}(u)\in H_{0}^{1}(\Omega_{\epsilon})^{\Gamma}\qquad\text{if
\ }u\in L^{\frac{2n}{n-2}}(\Omega_{\epsilon})^{\Gamma}.
\]

We rewrite problem $(\wp_{Q,\epsilon}^{\ast})$ in the following equivalent
way:
\begin{equation}
\left\{
\begin{array}
[c]{l}%
u=i_{\epsilon}^{\ast}\left[  Q(x)f(u)\right]  ,\\
u\in H_{0}^{1}(\Omega_{\epsilon}),
\end{array}
\right.  \label{defp}%
\end{equation}
where $f(s):=(s^{+})^{p}$ and $p:=\frac{n+2}{n-2}$.

We shall look for a solution to problem (\ref{defp}) of the form
\begin{equation}
u_{\epsilon}=V_{d,\eta}+\phi\text{\qquad with }(d,\eta)\in\Lambda^{\Gamma
}\text{ and }\phi\in K_{d,\eta}^{\epsilon,\bot}. \label{V}%
\end{equation}
As usual, our goal will be to find $(d,\eta)\in\Lambda^{\Gamma}$ and $\phi\in
K_{d,\eta}^{\epsilon,\bot}$ such that, for $\epsilon$ small enough,%
\begin{equation}
\Pi_{d,\eta}^{\epsilon,\bot}[V_{d,\eta}+\phi-i_{\epsilon}^{\ast}(Qf(V_{d,\eta
}+\phi))]=0 \label{ecpiort}%
\end{equation}
and%
\begin{equation}
\Pi_{d,\eta}^{\epsilon}[V_{d,\eta}+\phi-i_{\epsilon}^{\ast}(Qf(V_{d,\eta}%
+\phi))]=0. \label{ecpi}%
\end{equation}
First we will show that, for every $(d,\eta)\in\Lambda^{\Gamma}$ and
$\epsilon$ small enough, there exists an unique $\phi\in K_{d,\eta}%
^{\epsilon,\bot}$ which satisfies (\ref{ecpiort}). To this aim we consider the
linear operator $L_{d,\eta}^{\epsilon}:K_{d,\eta}^{\epsilon,\bot}\rightarrow
K_{d,\eta}^{\epsilon,\bot}$ defined by
\[
L_{d,\eta}^{\epsilon}(\phi):=\phi-\Pi_{d,\eta}^{\epsilon,\bot}\,i_{\epsilon
}^{\ast}[Qf^{\prime}(V_{d,\eta})\phi].
\]
It has the following properties.

\begin{proposition}
\label{invertibility} For every compact subset $D$ of $\Lambda^{\Gamma}$ there
exist $\epsilon_{0}>0$ and $c>0$ such that, for each $\epsilon\in
(0,\epsilon_{0})$ and each $(d,\eta)\in D,$
\begin{equation}
\left\Vert L_{d,\eta}^{\epsilon}(\phi)\right\Vert \geq c\left\Vert
\phi\right\Vert \text{\qquad for all }\phi\in K_{d,\eta}^{\epsilon,\bot},
\label{inv}%
\end{equation}
and the operator $L_{d,\eta}^{\epsilon}$ is invertible.
\end{proposition}

\begin{proof}
The argument given in \cite{Ge-Mu-Pi} to prove Lemma 5.1 carries over with
minor changes to our situation.
\end{proof}

The following estimates may be found in \cite{Li}.

\begin{lemma}
\label{lem1}For each $a,b,q\in\mathbb{R}$ with $a\geq0$ and $q\geq1$ there
exists a positive constant $c$ such that the following inequalities hold%
\[
\left\vert \left\vert a+b\right\vert ^{q}-a^{q}\right\vert \leq\left\{
\begin{array}
[c]{ll}%
c\min\{\left\vert b\right\vert ^{q},a^{q-1}\left\vert b\right\vert \} &
\text{if }0<q<1,\\
c(\left\vert a\right\vert ^{q-1}\left\vert b\right\vert +\left\vert
b\right\vert ^{q}) & \text{if }q\geq1.
\end{array}
\right.
\]

\end{lemma}

Again, the argument given to prove similar results in the literature carries
over with minor changes to prove the following result. We include it this time
to illustrate this fact and also because some of the estimates will be used
later on.

\begin{proposition}
\label{prop3} For every compact subset $D$ of $\Lambda^{\Gamma}$ there exist
$\epsilon_{0}>0$ and $c>0$ such that, for each $\epsilon\in(0,\epsilon_{0})$
and for each $(d,\eta)\in D,$ there exists a unique $\phi_{d,\eta}^{\epsilon
}\in K_{d,\eta}^{\epsilon,\bot}\subset H_{0}^{1}(\Omega_{\epsilon})^{\Gamma}$
which solves equation \emph{(\ref{ecpiort})} and satisfies
\begin{equation}
\left\Vert \phi_{d,\eta}^{\epsilon}\right\Vert \leq c\epsilon^{\frac{n-2}%
{n-1}}. \label{ecu21}%
\end{equation}
Moreover, the function $(d,\eta)\mapsto\phi_{d,\eta}^{\epsilon}$ is a $C^{1}$-map.
\end{proposition}

\begin{proof}
Note that $\phi\in K_{d,\eta}^{\epsilon,\bot}$ solves equation (\ref{ecpiort})
if and only if $\phi$ is a fixed point of the operator $T_{d,\eta}^{\epsilon
}:K_{d,\eta}^{\epsilon,\bot}\rightarrow K_{d,\eta}^{\epsilon,\bot}$ defined by%
\[
T_{d,\eta}^{\epsilon}(\phi)=(L_{d,\eta}^{\epsilon})^{-1}\Pi_{d,\eta}%
^{\epsilon,\bot}i_{\epsilon}^{\ast}\left[  Qf(V_{d,\eta}+\phi)-Qf^{\prime
}(V_{d,\eta})\phi-Q(\xi_{0})(\gamma_{0}U_{\delta,\xi})^{p}\right]  .
\]
We will prove that $T_{d,\eta}^{\epsilon}$ is a contraction on a suitable ball.

To this aim, we first show that there exist $\epsilon_{0}>0$ and $c>0$ such
that for, each $\epsilon\in(0,\epsilon_{0}),$
\begin{equation}
\left\Vert \phi\right\Vert \leq c\epsilon^{\frac{n-2}{n-1}}\quad
\Rightarrow\quad\left\Vert T_{d,\eta}^{\epsilon}(\phi)\right\Vert \leq
c\epsilon^{\frac{n-2}{n-1}}. \label{30}%
\end{equation}
From Proposition \ref{invertibility} we have that, for some $c>0$ and
$\epsilon$ small enough,%
\[
\left\Vert (L_{d,\eta}^{\epsilon})^{-1}\right\Vert \leq c\text{\qquad}%
\forall(d,\eta)\in D.
\]
Using (\ref{ecu1}) we obtain%
\begin{align*}
\left\Vert T_{d,\eta}^{\epsilon}(\phi)\right\Vert  &  \leq c\left\Vert
Q\left[  f(V_{d,\eta}+\phi)-f^{\prime}(V_{d,\eta})\phi\right]  -Q(\xi
_{0})(\gamma_{0}U_{\delta,\xi})^{p}\right\Vert _{\frac{2n}{n+2}}\\
&  \leq c\left\Vert Q\left[  f(V_{d,\eta}+\phi)-f(V_{d,\eta})-f^{\prime
}(V_{d,\eta})\phi\right]  \right\Vert _{\frac{2n}{n+2}}\\
&  \qquad+c\left\Vert Qf(V_{d,\eta})-Q(\gamma_{0}U_{\delta,\xi})^{p}%
\right\Vert _{\frac{2n}{n+2}}+c\gamma_{0}^{p}\left\Vert \left[  Q-Q(\xi
_{0})\right]  U_{\delta,\xi}^{p}\right\Vert _{\frac{2n}{n+2}}.
\end{align*}
Using the mean value theorem, Lemma \ref{lem1} and the H\"{o}lder inequality
we have that, for some $t\in(0,1)$,%
\begin{align*}
\left\Vert Q\left[  f(V_{d,\eta}+\phi)-f(V_{d,\eta})-f^{\prime}(V_{d,\eta
})\phi\right]  \right\Vert _{\frac{2n}{n+2}}  &  \leq c\left\Vert [f^{\prime
}(V_{d,\eta}+t\phi)-f^{\prime}(V_{d,\eta})]\phi\right\Vert _{\frac{2n}{n+2}}\\
&  \leq c\left\Vert f^{\prime}(V_{d,\eta}+t\phi)-f^{\prime}(V_{d,\eta
})\right\Vert _{n/2}\left\Vert \phi\right\Vert _{2^{\ast}}\\
&  \leq c(\left\Vert \phi\right\Vert _{2^{\ast}}+\left\Vert \phi\right\Vert
_{2^{\ast}}^{\frac{4}{n-2}})\left\Vert \phi\right\Vert _{2^{\ast}}\\
&  \leq c(\left\Vert \phi\right\Vert _{2^{\ast}}^{2}+\left\Vert \phi
\right\Vert _{2^{\ast}}^{p}).
\end{align*}
Moreover, using Lemma \ref{lemGeMuPi}\ one can show that%
\begin{align}
&  \left\Vert Qf(V_{d,\eta})-Q(\gamma_{0}U_{\delta,\xi})^{p}\right\Vert
_{\frac{2n}{n+2}}\leq c\left\Vert (P_{\epsilon}U_{\delta,\xi})^{p}%
-U_{\delta,\xi}^{p}\right\Vert _{\frac{2n}{n+2}}\nonumber\\
&  \leq\left(  c\int_{\Omega_{\epsilon}}\left\vert U_{\delta,\xi}%
^{p-1}(P_{\epsilon}U_{\delta,\xi}-U_{\delta,\xi})\right\vert ^{\frac{2n}{n+2}%
}+c\int_{\Omega_{\epsilon}}\left\vert P_{\epsilon}U_{\delta,\xi}-U_{\delta
,\xi}\right\vert ^{p+1}\right)  ^{\frac{n+2}{2n}}\label{power}\\
&  \leq c\delta,\nonumber
\end{align}
see inequality (6.4) in \cite{Ge-Mu-Pi}. Finally, setting $y=\frac{x-\xi
}{\delta}=\frac{x-\xi_{0}}{\delta}-\eta$\ and $\widetilde{\Omega}_{\epsilon
}:=\{y\in\mathbb{R}^{n}:\delta y+\xi\in\Omega_{\epsilon}\},$ and using the
mean value theorem, for some $t\in(0,1)$ we obtain%
\begin{align}
\left\Vert \left[  Q-Q(\xi_{0})\right]  U_{\delta,\xi}^{p}\right\Vert
_{\frac{2n}{n+2}}  &  =\left(  \int_{\widetilde{\Omega}_{\epsilon}}\left\vert
Q(\delta y+\delta\eta+\xi_{0})-Q(\xi_{0})\right\vert ^{\frac{2n}{n+2}}%
U^{p+1}(y)dy\right)  ^{\frac{n+2}{2n}}\nonumber\\
&  =\delta\left(  \int_{\widetilde{\Omega}_{\epsilon}}\left\vert \left\langle
\nabla Q(t\delta y+t\delta\eta+\xi_{0}),y+\eta\right\rangle \right\vert
^{\frac{2n}{n+2}}U^{p+1}(y)dy\right)  ^{\frac{n+2}{2n}}\label{Q}\\
&  \leq c\delta.\nonumber
\end{align}
This proves statement (\ref{30}).

Next we show that we may choose $\epsilon_{0}>0$ such that, for each
$\epsilon\in(0,\epsilon_{0}),$ the operator
\[
T_{d,\eta}^{\epsilon}:\{\phi\in K_{d,\eta}^{\epsilon,\bot}:\left\Vert
\phi\right\Vert \leq c\epsilon^{\frac{n-2}{n-1}}\}\rightarrow\{\phi\in
K_{d,\eta}^{\epsilon,\bot}:\left\Vert \phi\right\Vert \leq c\epsilon
^{\frac{n-2}{n-1}}\}
\]
is a contraction and, therefore, has a unique fixed point, as claimed.

If $\phi_{1},\phi_{2}\in\{\phi\in K_{d,\eta}^{\epsilon,\bot}:\left\Vert
\phi\right\Vert \leq c\epsilon^{\frac{n-2}{n-1}}\}$, using again the mean
value theorem we obtain%
\begin{align*}
\left\Vert T_{d,\eta}^{\epsilon}(\phi_{1})-T_{d,\eta}^{\epsilon}(\phi
_{2})\right\Vert  &  \leq c\left\Vert f(V_{d,\eta}+\phi_{1})-f(V_{d,\eta}%
+\phi_{2})-f^{\prime}(V_{d,\eta})(\phi_{1}-\phi_{2}))\right\Vert _{\frac
{2n}{n+2}}\\
&  =c\left\Vert [f^{\prime}(V_{d,\eta}+(1-t)\phi_{1}+\phi_{2})-f^{\prime
}(V_{d,\eta})](\phi_{1}-\phi_{2})\right\Vert _{\frac{2n}{n+2}}\\
&  \leq c\left\Vert f^{\prime}(V_{d,\eta}+(1-t)\phi_{1}+\phi_{2})-f^{\prime
}(V_{d,\eta})\right\Vert _{\frac{n}{2}}\left\Vert \phi_{1}-\phi_{2}\right\Vert
_{2^{\ast}}%
\end{align*}
for some $t\in\lbrack0,1],$ and arguing as before we conclude that
\begin{align*}
\left\Vert f^{\prime}(V_{d,\eta}+(1-t)\phi_{1}+\phi_{2})-f^{\prime}(V_{d,\eta
})\right\Vert _{\frac{n}{2}}  &  \leq c\left(  \left\Vert (1-t)\phi_{1}%
+\phi_{2}\right\Vert _{2^{\ast}}+\left\Vert (1-t)\phi_{1}+\phi_{2}\right\Vert
_{2^{\ast}}^{\frac{4}{n-2}}\right) \\
&  \leq c\left(  \left\Vert \phi_{1}\right\Vert _{2^{\ast}}+\left\Vert
\phi_{2}\right\Vert _{2^{\ast}}+\left\Vert \phi_{1}\right\Vert _{2^{\ast}%
}^{\frac{4}{n-2}}+\left\Vert \phi_{2}\right\Vert _{2^{\ast}}^{\frac{4}{n-2}%
}\right)
\end{align*}
Hence, if $\epsilon$ is sufficiently small, it follows that
\[
\left\Vert T_{d,\eta}^{\epsilon}(\phi_{1})-T_{d,\eta}^{\epsilon}(\phi
_{2})\right\Vert \leq\kappa\left\Vert \phi_{1}-\phi_{2}\right\Vert
\]
with $\kappa\in(0,1)$.

Finally, a standard argument shows that $(d,\eta)\mapsto\phi_{d,\eta
}^{\epsilon}$ is a $C^{1}$-map. This concludes the proof.
\end{proof}

Consider the functional $J_{\epsilon}:H_{0}^{1}(\Omega_{\epsilon}%
)\rightarrow\mathbb{R}$ defined by
\[
J_{\epsilon}(u):=\frac{1}{2}\int_{\Omega_{\epsilon}}|\nabla u|^{2}-\frac
{1}{p+1}\int_{\Omega_{\epsilon}}Q|u|^{p+1}.
\]
It is well known that the critical points of $J_{\epsilon}$ are the solutions
of problem (\ref{defp}). We define the reduced energy functional
$\widetilde{J}_{\epsilon}^{\Gamma}:\Lambda^{\Gamma}\rightarrow\mathbb{R}$ by
\begin{equation}
\widetilde{J}_{\epsilon}^{\Gamma}(d,\eta):=J_{\epsilon}(V_{d,\eta}%
+\phi_{d,\eta}^{\epsilon}). \label{fun1}%
\end{equation}
If $\Gamma=\{1\}$ is the trivial group, we simply write $\widetilde
{J}_{\epsilon}$ instead of $\widetilde{J}_{\epsilon}^{\Gamma}$ and $\Lambda$
instead of $\Lambda^{\Gamma}.$

Next we show that the critical points of $\widetilde{J}_{\epsilon}^{\Gamma}$
are $\Gamma$-invariant solutions of problem (\ref{defp}).

\begin{proposition}
\label{prop6}If $(d,\eta)\in\Lambda^{\Gamma}$ is a critical point of the
function $\widetilde{J}_{\epsilon}^{\Gamma},$ then $V_{d,\eta}+\phi_{d,\eta
}^{\epsilon}\in H_{0}^{1}(\Omega_{\epsilon})^{\Gamma}$ is a critical point of
the functional $J_{\epsilon}$ and, therefore, a $\Gamma$-invariant solution of
problem \emph{(\ref{defp})}.
\end{proposition}

\begin{proof}
Assume first that $\Gamma$ is the trivial group. Then $\Lambda=(0,\infty
)\times\mathbb{R}^{n}$ and the statement is proved using similar arguments to
those given to prove Lemma 6.1 in \cite{dfm2}\ or Proposition 2.2 in
\cite{Ge-Mu-Pi}.

If $\Gamma$ is an arbitrary closed subgroup of $O(n),$ then $\Lambda^{\Gamma}$
is the set of $\Gamma$-fixed points in $\Lambda$ of the action of $\Gamma$ on
the space $\mathbb{R}\times\mathbb{R}^{n}$ which is given by $g(t,x):=(t,gx)$
for $g\in\Gamma,$ $t\in\mathbb{R}$, $x\in\mathbb{R}^{n}.$ By the principle of
symmetric criticality \cite{Palais, Willem}, if $(d,\eta)\in\Lambda^{\Gamma}$
is a critical point of the function $\widetilde{J}_{\epsilon}^{\Gamma},$ then
$(d,\eta)$ is a critical point of $\widetilde{J}_{\epsilon}:(0,\infty
)\times\mathbb{R}^{n}\rightarrow\mathbb{R}$, and the result follows from the
previous case.
\end{proof}

\section{The asymptotic expansion of the reduced energy functional}

In order to find a critical point of $\widetilde{J}_{\epsilon}^{\Gamma}$ we
will use the following asymptotic expansion of the functional $\widetilde
{J}_{\epsilon}:(0,\infty)\times\mathbb{R}^{n}\rightarrow\mathbb{R}$.

\begin{proposition}
\label{lemaf}The asymptotic expansion
\[
\widetilde{J}_{\epsilon}(d,\eta)=c_{0}+Q(\xi_{0})^{-\frac{2}{p-1}}%
F(d,\eta)\epsilon^{\frac{n-2}{n-1}}+o(\epsilon^{\frac{n-2}{n-1}})
\]
holds true $C^{1}$-uniformly on compact subsets of $\Lambda$, where the
function $F:(0,\infty)\times\mathbb{R}^{n}\rightarrow\mathbb{R}$ is given by%
\begin{equation}
F(d,\eta):=\left\{
\begin{array}
[c]{ll}%
\alpha d+\beta\frac{1}{(1+\left\vert \eta\right\vert ^{2})d}-\gamma
\left\langle \frac{\nabla Q(\xi_{0})}{Q(\xi_{0})},\eta\right\rangle d &
\text{if }n=3,\medskip\\
\beta\left(  \frac{1}{(1+\left\vert \eta\right\vert ^{2})d}\right)
^{n-2}-\gamma\left\langle \frac{\nabla Q(\xi_{0})}{Q(\xi_{0})},\eta
\right\rangle d & \text{if }n\geq4.
\end{array}
\right.  \label{bpc}%
\end{equation}
for some positive constants $c_{0},\alpha,\beta$ and $\gamma$.
\end{proposition}

\begin{proof}
We write
\begin{align*}
J_{\epsilon}(V_{d,\eta}+\phi_{d,\eta}^{\epsilon})  &  =\frac{1}{2}\left\Vert
V_{d,\eta}+\phi_{d,\eta}^{\epsilon}\right\Vert ^{2}-\frac{1}{p+1}\int
_{\Omega_{\epsilon}}Q\left\vert V_{d,\eta}+\phi_{d,\eta}^{\epsilon}\right\vert
^{p+1}\\
&  =J_{\epsilon}(V_{d,\eta})+\gamma_{0}\int_{\Omega_{\epsilon}}(U_{\delta,\xi
}^{p}-\left(  P_{\epsilon}U_{\delta,\xi}\right)  ^{p})\phi_{d,\eta}^{\epsilon
}\\
&  -\gamma_{0}^{p}\int_{\Omega_{\epsilon}}\left[  Q-Q(\xi_{0})\right]  \left(
P_{\epsilon}U_{\delta,\xi}\right)  ^{p}\phi_{d,\eta}^{\epsilon}+\frac{1}%
{2}\left\Vert \phi_{d,\eta}^{\epsilon}\right\Vert ^{2}\\
&  -\frac{1}{p+1}\int_{\Omega_{\epsilon}}Q\left(  \left\vert V_{d,\eta}%
+\phi_{d,\eta}^{\epsilon}\right\vert ^{p+1}-\left\vert V_{d,\eta}\right\vert
^{p+1}-(p+1)V_{d,\eta}^{p}\phi_{d,\eta}^{\epsilon}\right)  .
\end{align*}
Then, using H\"{o}lder's inequality and inequalities (\ref{ecu21}),
(\ref{power}) and (\ref{Q}) we obtain%
\begin{align}
J_{\epsilon}(V_{d,\eta}+\phi_{d,\eta}^{\epsilon})  &  =J_{\epsilon}(V_{d,\eta
})+O\left(  \epsilon^{\frac{2(n-2)}{n-1}}\right) \nonumber\\
&  =\gamma_{0}^{2}\left[  \frac{1}{2}\int_{\Omega_{\epsilon}}U_{\delta,\xi
}^{p}\left(  P_{\epsilon}U_{\delta,\xi}\right)  -\frac{1}{p+1}\int
_{\Omega_{\epsilon}}\left\vert P_{\epsilon}U_{\delta,\xi}\right\vert
^{p+1}\right] \label{red}\\
&  -\frac{1}{p+1}\gamma_{0}^{p+1}\int_{\Omega_{\epsilon}}\left[  Q-Q(\xi
_{0})\right]  \left\vert P_{\epsilon}U_{\delta,\xi}\right\vert ^{p+1}+O\left(
\epsilon^{\frac{2(n-2)}{n-1}}\right)  .\nonumber
\end{align}
Next, we compute the first summand on the right-hand side of equality
(\ref{red}). From Lemma \ref{lemGeMuPi} we have that%
\begin{align*}
&  \frac{1}{2}\int_{\Omega_{\epsilon}}U_{\delta,\xi}^{p}\left(  P_{\epsilon
}U_{\delta,\xi}\right)  -\frac{1}{p+1}\int_{\Omega_{\epsilon}}\left\vert
P_{\epsilon}U_{\delta,\xi}\right\vert ^{p+1}\\
&  =\frac{p-1}{2(p+1)}\int_{\Omega_{\epsilon}}U_{\delta,\xi}^{p+1}-\frac{1}%
{2}\int_{\Omega_{\epsilon}}U_{\delta,\xi}^{p}\left(  P_{\epsilon}U_{\delta
,\xi}-U_{\delta,\xi}\right)  -\frac{1}{p+1}\int_{\Omega_{\epsilon}}\left\vert
\left\vert P_{\epsilon}U_{\delta,\xi}\right\vert ^{p+1}-U_{\delta,\xi}%
^{p+1}\right\vert \\
&  =\frac{p-1}{2(p+1)}\int_{\Omega_{\epsilon}}U_{1,0}^{p+1}-\frac{1}{2}%
\int_{\Omega_{\epsilon}}U_{\delta,\xi}^{p}\left(  P_{\epsilon}U_{\delta,\xi
}-U_{\delta,\xi}\right)  +o(\epsilon^{\frac{n-2}{n-1}})\\
&  =\frac{p-1}{2(p+1)}\int_{\mathbb{R}^{n}}U_{1,0}^{p+1}+\frac{1}{2}%
\int_{\mathbb{R}^{n}}U_{\delta,\xi}^{p}\Upsilon_{\delta,\xi}^{\epsilon
}+o(\epsilon^{\frac{n-2}{n-1}}),
\end{align*}
where
\begin{equation}
\Upsilon_{\delta,\xi}^{\epsilon}(x):=\alpha_{n}\delta^{\frac{n-2}{2}}%
H(x,\xi)+\alpha_{n}\frac{1}{\delta^{\frac{n-2}{2}}(1+|\eta|^{2})^{\frac
{n-2}{2}}}\frac{\epsilon^{n-2}}{|x-\xi_{0}|^{n-2}}. \label{ecu41}%
\end{equation}
Setting $x=\xi+\delta y$ we have
\begin{align*}
&  \alpha_{n}\int_{\mathbb{R}^{n}}U_{\delta,\xi}^{p}\Upsilon_{\delta,\xi
}^{\epsilon}\\
&  =\alpha_{n}\int_{\mathbb{R}^{n}}U_{\delta,\xi}^{p}(x)(\delta^{\frac{n-2}%
{2}}H(x,\xi))dx+\alpha_{n}\int_{\mathbb{R}^{n}}U_{\delta,\xi}^{p}(x)\left(
\frac{1}{\delta^{\frac{n-2}{2}}(1+|\eta|^{2})^{\frac{n-2}{2}}}\frac
{\epsilon^{n-2}}{|x-\xi_{0}|^{n-2}}\right)  dx\\
&  =\alpha_{n}\delta^{n-2}\int_{\mathbb{R}^{n}}U_{1,0}^{p}(y)H(\delta
y+\delta\eta+\xi_{0},\delta\eta+\xi_{0})dy\\
&  +\alpha_{n}\frac{1}{(1+|\eta|^{2})^{\frac{n-2}{2}}}\int_{\mathbb{R}^{n}%
}U_{1,0}^{p}(y)\left(  \frac{\epsilon^{n-2}}{\delta^{n-2}|y-\eta|^{n-2}%
}\right)  dy\\
&  =\alpha_{n} \left(  \int_{\mathbb{R}^{n}}U_{1,0}^{p}\right)  H(\xi_{0}%
,\xi_{0})\delta^{n-2}(1+o(1))+\alpha_{n}g(\eta)\frac{1}{\delta^{n-2}}%
\epsilon^{n-2}(1+o(1)),
\end{align*}
where the function $g:\mathbb{R}^{n}\rightarrow\mathbb{R}$ is defined by
\[
g(\eta):={\frac{1}{(1+|\eta|^{2})^{\frac{n-2}{2}}}\int_{\mathbb{R}^{n}}%
\frac{1}{|y-\eta|^{n-2}}}U_{1,0}^{p}(y)dy.
\]
Since $-\Delta U=U^{p}$ in $\mathbb{R}^{n},$ an easy computation shows that
\[
g(\eta)={\frac{1}{(1+|\eta|^{2})^{\frac{n-2}{2}}}}U_{1,0}(\eta)=\alpha
_{n}{\frac{1}{(1+|\eta|^{2})^{n-2}}}.
\]
To compute the second summand on the right-hand side of equality (\ref{red})
we use the Taylor expansion
\[
Q(\delta y+\xi_{0}+\delta\eta)=Q(\xi_{0})+\delta\langle\nabla Q(\xi
_{0}),y+\eta\rangle+O(\delta^{2}(1+|y|^{2}))
\]
to obtain
\begin{align*}
&  \int_{\Omega_{\epsilon}}\left[  Q-Q(\xi_{0})\right]  |P_{\epsilon}%
U_{\delta,\xi}|^{p+1}=\int_{\Omega_{\epsilon}}\left[  Q-Q(\xi_{0})\right]
U_{\delta,\xi}^{p+1}+o(\epsilon^{\frac{n-2}{n-1}})\\
&  =\int_{\widetilde{\Omega}_{\epsilon}}(Q(\delta y+\xi_{0}+\delta\eta
)-Q(\xi_{0}))U_{1,0}^{p+1}(y)dy+o(\epsilon^{\frac{n-2}{n-1}})\\
&  =\delta\int_{\mathbb{R}^{n}}\langle\nabla Q(\xi_{0}),\eta\rangle
U_{1,0}^{p+1}(y)dy+\delta\int_{\mathbb{R}^{n}}\frac{\langle\nabla Q(\xi
_{0}),y\rangle}{(1+|y|^{2})^{n}}dy+O\left(  \epsilon^{\frac{2(n-2)}{n-1}%
}\right) \\
&  =\delta\langle\nabla Q(\xi_{0}),\eta\rangle\left(  \int_{\mathbb{R}^{n}%
}U_{1,0}^{p+1}\right)  (1+o(1)),
\end{align*}
because $\int_{\mathbb{R}^{n}}\frac{\langle\nabla Q(\xi_{0}),y\rangle
}{(1+|y|^{2})^{n}}dy=0$. Collecting all the previous information we obtain%
\begin{align*}
&  \widetilde{J}_{\epsilon}(d,\eta)=J_{\epsilon}(V_{d,\eta}+\phi_{d,\eta
}^{\epsilon})\\
&  =\left\{
\begin{array}
[c]{ll}%
c_{0}+\gamma_{0}^{2}\left(  c_{1}H(\xi_{0},\xi_{0})d+c_{2}g(\eta)\frac{1}%
{d}-c_{3}\left\langle \frac{\nabla Q(\xi_{0})}{Q(\xi_{0})},\eta\right\rangle
d\right)  \sqrt{\epsilon}+o(\sqrt{\epsilon}) & \text{if }n=3,\medskip\\
c_{0}+\gamma_{0}^{2}\left(  c_{2}g(\eta)\frac{1}{d^{n-2}}-c_{3}\left\langle
\frac{\nabla Q(\xi_{0})}{Q(\xi_{0})},\eta\right\rangle d\right)
\epsilon^{\frac{n-2}{n-1}}+o(\epsilon^{\frac{n-2}{n-1}}) & \text{if }n\geq4,
\end{array}
\right.
\end{align*}
as claimed.\medskip
\end{proof}

\noindent\textbf{Proof of theorem \ref{ts}.\quad}We will show that the
function $F$ defined in (\ref{bpc}) has a critical point $(d_{0},\eta_{0}%
)\in\Lambda^{\Gamma}=(0,\infty)\times\left(  \mathbb{R}^{n}\right)  ^{\Gamma}$
which is stable under $C^{1}$-perturbations. Then, we deduce from Proposition
\ref{lemaf} that the functional $\widetilde{J}_{\epsilon}^{\Gamma}$ has a
critical point in $\Lambda^{\Gamma}$ for $\epsilon$ small enough, so the
result follows from Proposition \ref{prop6}.

Let $n=3.$ Set $\zeta_{0}:=\frac{\nabla Q(\xi_{0})}{Q(\xi_{0})}\ $and consider
the half space $\mathcal{H}:=\{\eta\in\mathbb{R}^{3}:\alpha-\gamma\langle
\zeta_{0},\eta\rangle>0\}$. For each $\eta\in\mathcal{H}$ there exists a
unique $d=d(\eta),$ given by
\[
d(\eta)=\sqrt{\frac{\beta}{(1+|\eta|^{2})(\alpha-\gamma\langle\zeta_{0}%
,\eta\rangle)}}\in(0,\infty),
\]
such that $F_{d}(d,\eta)=0$. Moreover, $F_{dd}(d(\eta),\eta)>0$ for any
$\eta\in\mathcal{H}$. Consider the function $\widetilde{F}:\mathcal{H}%
\rightarrow\mathbb{R}$ defined by
\[
\widetilde{F}(\eta):=F(d(\eta),\eta)=2\beta^{2}\sqrt{\frac{\alpha
-\gamma\langle\zeta_{0},\eta\rangle}{1+|\eta|^{2}}.}%
\]
The point
\[
\eta_{0}:=\left(  \frac{\alpha-\sqrt{\alpha^{2}+\gamma^{2}\left\vert \zeta
_{0}\right\vert ^{2}}}{\gamma\left\vert \zeta_{0}\right\vert ^{2}}\right)
\zeta_{0}%
\]
is a strict maximum point of $\widetilde{F}.$ Setting $d_{0}:=d(\eta_{0})$ we
deduce from Lemma 5.7 in \cite{Mo-Pi} that $(d_{0},\eta_{0})$ is a $C^{1}%
$-stable critical point of the function $F$. Note that, since $\xi_{0}%
\in\Omega^{\Gamma}$ and $Q$ is $\Gamma$-invariant, $\nabla Q(\xi_{0}%
)\in\left(  \mathbb{R}^{n}\right)  ^{\Gamma}.$ Hence, $(d_{0},\eta_{0}%
)\in\Lambda^{\Gamma}.$

If $n\geq4$ arguing as in the previous case we easily conclude that, if
\[
\eta_{0}:=-{\frac{\nabla Q(\xi_{0})}{|\nabla Q(\xi_{0})|},}\qquad
d_{0}:=\left(  {\frac{(n-2)\beta}{2^{n-2}\gamma}}{\frac{Q(\xi_{0})}{|\nabla
Q(\xi_{0})|}}\right)  ^{{\frac{1}{n-1}}},
\]
then $(d_{0},\eta_{0})$ is a $C^{1}$-stable critical point of the function $F$
and $(d_{0},\eta_{0})\in\Lambda^{\Gamma}.$ This concludes the proof.
\ \hfill$\square$

\section{Final remarks}

One may wonder whether Proposition \ref{ps} is also true in other dimensions.
We show that this is not so.

If $N=k_{1}+k_{2}$ we write the elements of $\mathbb{R}^{N}$ as $(y_{1}%
,y_{2})$ with $y_{i}\in\mathbb{R}^{k_{i}},$ and the elements of $\mathbb{R}%
^{m+1}$ as $(t,\zeta)$ with $t\in\mathbb{R},$ $\zeta\in\mathbb{R}^{m}.$

\begin{proposition}
\label{ps+}Let $N=k_{1}+k_{2},$ $\mathcal{D}$ be an $\left[  O(k_{1})\times
O(k_{2})\right]  $-invariant bounded smooth domain in $\mathbb{R}^{N}$ such
that $0\notin\overline{\mathcal{D}},$ and $f\in C^{0}(\mathbb{R}).$ Set
\[
\mathcal{U}:=\{(t,\zeta)\in\mathbb{R}\times\mathbb{R}^{m}:\mathfrak{h}%
_{\mathbb{R}}(\left\vert y_{1}\right\vert ,\left\vert y_{2}\right\vert
)=(t,\left\vert \zeta\right\vert )\text{ for some }(y_{1},y_{2})\in
\mathcal{D}\}
\]
and let $u\in C^{2}(\mathcal{U)}$, $u(t,\zeta)=\mathfrak{u}(t,\left\vert
\zeta\right\vert ),$ be an $O(m)$-invariant solution of equation%
\begin{equation}
-\Delta u=\frac{1}{2\left\vert x\right\vert }f(u) \label{eqB}%
\end{equation}
in $\mathcal{U}$. Then $v(y_{1},y_{2}):=\mathfrak{u}(\mathfrak{h}_{\mathbb{R}%
}(\left\vert y_{1}\right\vert ,\left\vert y_{2}\right\vert ))$ is an $\left[
O(k_{1})\times O(k_{2})\right]  $-invariant solution of equation%
\begin{equation}
-\Delta v=f(v) \label{eqA}%
\end{equation}
in $\mathcal{D}$ if and only if $k_{1}=k_{2}=m.$
\end{proposition}

\begin{proof}
A straighforward computation shows that a function $v(y_{1},y_{2}%
)=\mathfrak{v}(\left\vert y_{1}\right\vert ,\left\vert y_{2}\right\vert )$
solves equation (\ref{eqA}) in $\{(y_{1},y_{2})\in\mathcal{D}:y_{1}\neq0,$
$y_{2}\neq0\}$ if and only if $\mathfrak{v}$ solves
\begin{equation}
-\Delta\mathfrak{v}-\frac{k_{1}-1}{z_{1}}\frac{\partial\mathfrak{v}}{\partial
z_{1}}-\frac{k_{2}-1}{z_{2}}\frac{\partial\mathfrak{v}}{\partial z_{2}%
}=f(\mathfrak{v}) \label{eq1}%
\end{equation}
in $\mathcal{D}_{0}:=\{z=(z_{1},z_{2})\in\mathbb{R}^{2}:z_{1},z_{2}>0,$
$z_{1}=\left\vert y_{1}\right\vert ,$ $z_{2}=\left\vert y_{2}\right\vert $,
$(y_{1},y_{2})\in\mathcal{D}\}.$

Similarly, a function $u(t,\zeta)=\mathfrak{u}(t,\left\vert \zeta\right\vert
)$ solves equation (\ref{eqB}) in $\{(t,\zeta)\in\mathcal{U}:\zeta\neq0\}$ if
and only if $\mathfrak{u}$ solves%
\begin{equation}
-\Delta\mathfrak{u}-\frac{m-1}{x_{2}}\frac{\partial\mathfrak{u}}{\partial
x_{2}}=\frac{1}{2\left\vert x\right\vert }f(\mathfrak{u}) \label{eq2}%
\end{equation}
in $\mathcal{U}_{0}:=\{x=(x_{1},x_{2})\in\mathbb{R}^{2}:x_{2}>0,$
$x_{2}=\left\vert \zeta\right\vert ,$ $(x_{1},\zeta)\in\mathcal{U}\}.$

Assuming that $\mathfrak{u}$ solves equation (\ref{eq2}) in $\mathcal{U}_{0}%
$,\ we show next that $\mathfrak{v}:=\mathfrak{u}\circ\mathfrak{h}%
_{\mathbb{R}}$ solves equation (\ref{eq1}) in $\mathcal{D}_{0}$\ if and only
if $k_{1}=k_{2}=m.$ A straightforward computation yields%
\begin{align*}
&  -\Delta\mathfrak{v}-\frac{k_{1}-1}{z_{1}}\frac{\partial\mathfrak{v}%
}{\partial z_{1}}-\frac{k_{2}-1}{z_{2}}\frac{\partial\mathfrak{v}}{\partial
z_{2}}\\
&  =2\left\vert z\right\vert ^{2}\left(  -\Delta\mathfrak{u}-\left[
\frac{k_{1}-1}{\left\vert z\right\vert ^{2}}-\frac{k_{2}-1}{\left\vert
z\right\vert ^{2}}\right]  \frac{\partial\mathfrak{u}}{\partial x_{1}}-\left[
\frac{k_{1}-1}{\left\vert z\right\vert ^{2}}\frac{z_{2}}{z_{1}}+\frac{k_{2}%
-1}{\left\vert z\right\vert ^{2}}\frac{z_{1}}{z_{2}}\right]  \frac
{\partial\mathfrak{u}}{\partial x_{2}}\right)  .
\end{align*}
Note that $\left\vert z\right\vert ^{2}=\left\vert \mathfrak{h}_{\mathbb{R}%
}(z)\right\vert .$ So if $\mathfrak{u}$ solves equation (\ref{eq2}),\ we have
that%
\[
-\Delta\mathfrak{v}-\frac{k_{1}-1}{z_{1}}\frac{\partial\mathfrak{v}}{\partial
z_{1}}-\frac{k_{2}-1}{z_{2}}\frac{\partial\mathfrak{v}}{\partial z_{2}%
}=|\mathfrak{v}|^{p-2}\mathfrak{v}%
\]
if and only if%
\[
\frac{k_{1}-1}{\left\vert z\right\vert ^{2}}=\frac{k_{2}-1}{\left\vert
z\right\vert ^{2}}\text{\qquad and\qquad}\frac{k_{1}-1}{\left\vert
z\right\vert ^{2}}\frac{z_{2}}{z_{1}}+\frac{k_{2}-1}{\left\vert z\right\vert
^{2}}\frac{z_{1}}{z_{2}}=\frac{m-1}{z_{1}z_{2}}%
\]
if and only if $k_{1}=k_{2}=m.$
\end{proof}

The argument given in \cite{ps} to prove Proposition \ref{ps} uses polar
coordinates. Note that if we write $z_{1}=r\cos\theta,$ $z_{2}=r\sin\theta,$
$x_{1}=\rho\cos\varphi,$ $x_{2}=\rho\sin\varphi,$ then the Hopf map
$x=\frac{1}{2}\mathfrak{h}_{\mathbb{R}}(z)$ becomes%
\[
\rho=\frac{1}{2}r^{2},\qquad\varphi=2\theta,
\]
which is the map considered in \cite{ps}.

\end{document}